\documentclass[11pt]{amsart}
\usepackage{amssymb,amsmath,cite,hyperref}
\usepackage[usenames]{color} \newtheorem{theorem}{Theorem}[section]
\newtheorem{proposition}[theorem]{Proposition}
\newtheorem{lemma}[theorem]{Lemma}
\newtheorem{definition}[theorem]{Definition}

\newtheorem{example}[theorem]{Example}

\newtheorem{claim}[theorem]{Claim}
\newtheorem{conjecture}[theorem]{Conjecture}
\newtheorem{remark}[theorem]{Remark}
\newtheorem{problem}[theorem]{Problem}

\DeclareMathOperator{\sh}{Shift}

\title{Tur\'{a}n, involution and shifting}
\author{
Gil Kalai and Eran Nevo}
\address{Einstein Institute of Mathematics, The Hebrew University of Jerusalem, Jerusalem, Israel}
\email{kalai@math.huji.ac.il}
\address{Einstein Institute of Mathematics, The Hebrew University of Jerusalem, Jerusalem, Israel}
\email{nevo@math.huji.ac.il}
\thanks{Research of Kalai is partially supported by ERC advanced grant 320924, BSF grant 2006066, and NSF grant DMS-1300120,
and of Nevo by Israel Science Foundation
grant ISF-1695/15, by grant 2528/16 of the ISF-NRF Singapore joint research program, and by ISF-BSF joint grant 2016288.
}
\begin{document}
\maketitle
\section{Introduction}
Let $T(n)$ be a graph on the vertex set $[n]=\{1,2,\ldots,n\}$ which is the disjoint union of two cliques of
sizes $\lfloor\frac{n}{2}\rfloor$ and $\lceil\frac{n}{2}\rceil$.
Our starting point is the following influential Mantel-Tur\'{a}n theorem~\cite{Mantel, Turan}.

\begin{theorem}[Mantel, Tur\'{a}n]\label{thm:Turan}
Let $G$ be a graph on $[n]$ where every $3$-subset of $[n]$ contains an edge of $G$. Then the edge sets satisfy $$|E(G)|\geq |E(T(n))|.$$
\end{theorem}

Let $B(n)$ be the set of edges on $[n]$
\[B(n)=\{ab: 1\leq a<b\leq n, a+b\leq n\}.\]
Note that $|B(n)|=|E(T(n))|$, and that $B(n)$ is a \emph{shifted} graph, a.k.a. threshold graph.
(By a slight abuse of notation we refer to $B(n)$ also as the graph with the vertex set $[n]$ and edge set $B(n)$.)
However, not every $3$-subset of $[n]$ contains an edge of $B(n)$.
We relate $B(n)$ to the Mantel-Tur\'{a}n theorem in two ways:
first, by weakening the condition in Theorem~\ref{thm:Turan} so that
its conclusion applies to $B(n)$ as well; second, by strengthening
the conclusion in Theorem~\ref{thm:Turan} to an algebraic one,
namely that any graph $G$ as in the theorem \emph{dominates} $B(n)$.
The later, as we shall see, serves as an inspiration for
an algebraic approach to Tur\'{a}n's $(3,4)$-problem.

\begin{theorem}[Tur\'{a}n with involution]\label{thm:pairs}
Let $G$ be a graph on $[n]$.
Assume there exists an involution $\tau$ on $[n]$
such that any $3$-set $S\subseteq [n]$ satisfies
\begin {equation}
\label {e:inv}
 |E(G[S])|+|E(G[\tau(S)])|\geq 2.
\end {equation}

Then $|E(G)|\geq |E(T(n))|$.
\end{theorem}

Here $G[A]$ denotes the induced subgraph of $G$ on the subset $A$ of the vertex set.
First, note that the condition in Theorem~\ref{thm:Turan}
(``any triple supports an edge") implies
the condition in the above theorem (``any triple and its $\tau$-image
support two edges", counted with repetitions). In fact, if $<\tau>=\mathbb{Z}_2$ acts
trivially on $[n]$, then the conditions in both theorems coincide.

Second, if $\tau$ acts on $[n]$ with at most one fixed point, then by applying a suitable
permutation of the vertices we can assume $\tau(i)=n+1-i$ for any $i\in [n]$.
For this $\tau$, $(B(n),\tau)$ satisfies the assumptions of Theorem~\ref{thm:pairs}.

We now turn to an algebraic statement that
implies Theorem~\ref{thm:Turan}. First we define 
the notion of domination. 

Let $X=(x_{i,j})$ be an $n\times n$ matrix of variables,
and $C_k(X)=(c_{S,T})$ its $k$-th \emph{compound} matrix,
namely the $\binom{n}{k}\times \binom{n}{k}$ matrix, where for $k$-subsets $S,T$
of $[n]$, $c_{S,T}$ equals the determinant of the $(S,T)$-minor of $X$,
computed in the field extension $\mathbb{Q}(x_{i,j})$ over the rationales say.
Let $F_1$ and $F_2$ be two families of $k$-subsets of $[n]$.
Then $F_1$ \emph{dominates} $F_2$ if the submatrix of  $C_k(X)$ with rows indexed
by $F_2$ and columns indexed by $F_1$ has rank $|F_2|$.
This implies, of course, that $|F_2| \le |F_1|$.

\begin{theorem}[Tur\'{a}n with domination]\label{thm:domination_Turan}
Let $G$ be a graph on $[n]$ where every $3$-subset of $[n]$ contains
an edge of $G$. Then $E(G)$ dominates $B(n)$.
\end{theorem}

We prove this statement via exterior algebraic shifting,
introduced by Kalai~\cite{Kalai_ext_shift},
and via relations between algebraic
and combinatorial shifting established by Hibi and Murai\cite{Murai-Hibi, Murai-Join}.
(We can apply algebraic shifting w.r.t.
any fixed term order $<_t$ that satisfies $ab<_t a'b'$
whenever $a+b<a'+b'$.)

We do not know if there is a common generalization of Theorems~\ref{thm:pairs}
and \ref{thm:domination_Turan}.

\begin{problem}
Let $G$ be a graph on $[n]$.
Assume there exists an involution $\tau$ on $[n]$
such that any $3$-set $S\subseteq [n]$ satisfies relation (\ref {e:inv}).
Must $E(G)$ dominate $B(n)$?
\end{problem}


\subsection {Tur\'{a}n's $(3,4)$-conjecture}

Theorem~\ref{thm:domination_Turan} is motivated by an algebraic
approach to a famous conjecture by Tur\'{a}n \cite{Turan, Turan-problems}:
Turan asked for the minimum number of edges in a 3-uniform hypergraph on $[n]$ such that every four
vertices span at least one edge. (Equivalently he asked for the
maximum number of edges if every four vertices span at most three edges.)
Tur\'an conjectured in 1940 that the minimum is attained by the following hypergraph,
now called the {\it Tur\'an's (3,4)-hypergraph}:

Divide $[n]$ to three equal-as-possible sets $A_1,A_2,A_3$ and consider all 3-sets which are
either contained in $A_i$ or contain two elements from $A_i$ and one element
from $A_{i+1 (\mod 3)}$, $i=1,2,3$. Thus Turan's conjecture is:

\begin{conjecture}[Tur\'{a}n's $(3,4)$-conjecture]\label{conj:(3,4)}
  Let $H$ be a collection of $3$-sets on $[n]$ such that any $4$-subset on $[n]$
supports a $3$-set in $H$. Then $|H|\geq h(n)$,
where
$$h(n):=\binom{s_1}{3}+\binom{s_2}{3}+\binom{s_3}{3}
  +s_1\binom{s_2}{2}+s_2\binom{s_3}{2}+s_3\binom{s_1}{2},~~~~ n=s_1+s_2+s_3, $$
and $|s_i-s_j|\leq 1$ for any $1\leq i,j\leq 3$.
\end{conjecture}

For more background on this
problem see~\cite{34A-Keevash2011,34B-deCaen1994,34D-Kostochka1982,
34E-Brown1983,34F-Fon-Der-Flaass1988,34G-Cung-Lu1999,34H-Frohmader2008,
34I-Razborov2010,34J-Pihurko2011}.

Consider, now,  the collection of triples
\[C(n):=\{abc:\ a+c\leq n,\ 2a+b\leq n,\ 1\leq a<b<c\leq n\}.\]
Then $C(n)$ is a shifted family, and we expect it to play a
similar role for Turan's (3,4)-conjecture as the
role of $B(n)$ for the Mantel-Turan's theorem. Indeed, we note first
that
\begin{equation}
|C(n)|=h(n)=\max_s s \cdot {{n-s-1} \choose {2}} =
\lfloor\frac{n}{3}\rfloor\binom{n-\lfloor\frac{n}{3}\rfloor-1}{2}.
\end{equation}

The following conjecture implies Tur\'an's $(3,4)$-Conjecture.
It also strengthenes Kalai's conjecture from \cite[Eq.(3)]{Kalai:ResearchProblem}.

\begin{conjecture}\label{conj:(3,4)domination}
Let $H$ be a 3-uniform hypergraph on $[n]$ such that any four vertices span an edge.
Then $H$ dominates $C(n)$.
\end{conjecture}

Conjecture~\ref{conj:(3,4)domination} was verified, by computer, for all $H$ arising
from Kostochka's construction \cite{34D-Kostochka1982} with $n=3k\le 18$ vertices ($|H|=h(n)$).


Outline: Section~\ref{s:bacground}
contains background on dominance and shifting.
We prove Theorem~\ref{thm:pairs} in Section~\ref{sec:pairs}, and  Theorem~\ref{thm:domination_Turan} in Section~\ref{sec:algebraic}. In Section~\ref{s:34}
we discuss Tur\'an's $(3,4)$-conjecture.

\section*{Acknowledgements} We are grateful to Sonia Balagopalan for conducting computer experiments to test Conjecture~\ref{conj:(3,4)domination} and related conjectures.

\section{Dominance and  shifting}
\label{s:bacground}

\subsection{Dominance}

Let $X=(x_{ij})_{1 \le i, j \le n}$ be a matrix of $n^2$ variables.
Recall that the \emph{$k$-th compound matrix} $C_k(X)$ is the matrix of $k$ by $k$ minors, namely,
$$C_k(X)=(c_{S,T})_{S,T \in {{[n]} \choose {k}}},$$ where $$c_{ST}=det(x_{ij})_{i \in S,j\in T}.$$
(The order of rows and columns is not important in this subsection; it will be in the next).

\begin{definition}
Given two $k$-uniform hypergraphs on the vertex set $[n]$ $E_1$ and $E_2$,

(i) $E_1$ \emph{dominates} $E_2$ if the matrix $C_{E_1,E_2}(X)$, whose rows and columns correspond
to sets in $E_2$ and $E_1$ respectively,
has rank $|E_2|$.

(ii) $E_1$ and $E_2$ are \emph{weakly isomorphic} if each dominates the other.

\end{definition}

Of course, if $E_2$ dominates $E_1$ then  $|E_2| \ge |E_1|$, and if $E_1$ and $E_2$ are
weakly isomorphic then $|E_1|=|E_2|$.

One observes that if $E_1$ and $E_2$ are combinatorially isomorphic then they are weakly isomorphic.
This relies on the condition that the matrix $X$ is generic; compare with the Permutation Lemma of \cite{BK}.

\begin{example}
  A graph $G$ on $[n]$ dominates a star on $[n]$ iff $G$ is connected, and $G$ is dominated by a star on $[n]$ iff $G$ contains no cycle. In particular, $G$ is weakly isomorphic to a star on $[n]$ iff $G$ is a tree.
\end{example}
This example is the case $k=2$ of Proposition~\ref{prop:Tree_dom} below.

However, as we shall see, weak isomorphism is not a transitive relation.

\begin{example}\label{ex:not_transitive}

Let $C_5$ be a  cycle of length five. We will see later that $C_5$ is weakly isomorphic to
the graphs $G_1$ and $G_2$ which are themselves not weakly isomorphic, where
$G_1=\{ 12,13,14,15,23\}$ and $G_2=\{ 12,13,14,23,24\}$.

\end{example}

The following problem naturally arises:

\begin{problem}\label{prob:transitive}
What is the transitive closure of weak isomorphism for $k$-uniform hypergraph on $n$ vertices?
\end{problem}

\subsection{Connection with homology.}
Given a $k$-uniform hypergraph $G$ let $K(G)$ be the $(k-1)$-dimensional
simplicial complex whose $(k-1)$-faces are the edges in $G$, with complete $(k-2)$-dimensional skeleton.

 The $i$-th \emph{reduced} rational homology group of a simplicial complex $K$ is denoted by $\tilde{H}_{i}(K,\mathbb{Q})$.

A $k$-uniform hypergraph $G$ whose edges are all the $k$-subsets of $[n]$ containing some fixed element $v\in [n]$ is called a \emph{(spanning) $k$-star} on $[n]$; a $2$-star is simply a star, in the graph theoretical sense.
\begin{proposition}\label{prop:Tree_dom}

(i) A $k$-uniform hypergraph $G$ on $[n]$ dominates a (spanning) $k$-star on $[n]$ if and only if $\tilde{H}_{k-2}(K(G),\mathbb{Q}) = 0$.

(ii) A $k$-star on $[n]$ dominates a $k$-uniform hypergraph $G$ if and only if $\tilde{H}_{k-1}(G,\mathbb{Q})=0$.


\end{proposition}
For a proof one can either argue directly, or, as we shall do, use known properties of algebraic shifting, discussed next. The proof is postponed then to the next subsection.

\subsection{Algebraic shifting}\label{subsec:alg_shift}

\begin{definition}

(i) The {\it partial order} $<$ on $\mathcal{F}\subseteq {[n] \choose {k}}$ is defined as follows: If
$S=\{s_1,s_2,\dots,s_k\}$, $s_1<s_2<\cdots<s_k$ and $T=\{t_1,t_2,\dots,t_k\}$, $t_1<t_2<\cdots<t_k$, then
$S<T$ if $s_1\le t_1, s_2\le t_2 ,\dots ,s_k\le t_k$.

(ii)
A \emph{term order} $<_t$
 on ${[n] \choose k}$ is
a linear extension of the partial order $<$.
\end{definition}


\begin {definition}

A family $\mathcal{F}\subseteq {[n] \choose {k}}$
is \emph{shifted} if it is closed down under the partial order, namely, for any $F\in \mathcal{F}$ such
that $i<j\in F$ and $i\notin F$, the set $S=(F\setminus \{j\})\cup\{i\}$ is in $\mathcal{F}$.

\end {definition}

For example, $G_1$ and $G_2$ of Example~\ref{ex:not_transitive} are shifted. Shifted graphs
are called {\it threshold graphs}.

\begin {definition}
Let $A$ be an invertible $n\times n$ matrix over some
field, 
and let  $<_t$ be a term order on  ${{[n]} \choose {k}}$.

For $\mathcal{F} \subseteq {[n] \choose k}$, the \emph{$(<_t,A)$-exterior shifting} $\mathcal{F}^{<_t,A}$  of $\mathcal{F}$
is the $<_t$-smallest family $\mathcal{G}\subseteq {[n] \choose {k}}$ of size $|\mathcal{F}|$
such that $\mathcal{G}$ is weakly isomorphic to $\mathcal{F}$ w.r.t. $C_k(A)$.
(For two same size families of $k$-subsets of $[n]$, $Y$ and $Z$, say $Y$ is $<_t$-smaller than $Z$
iff the least element of the symmetric difference $Y\bigtriangleup Z$ w.r.t. $<_t$ belongs to $Y$.)
\end {definition}

Clearly, such $\mathcal{G}$ exists, as $\mathcal{F}$ is weakly isomorphic to itself.

Equivalently, start with $\mathcal{G}$ empty, and
add elements to $\mathcal{G}$ one-by-one ``greedily"
according to the order $<_t$ if the rank of the corresponding
matrix $C_{\mathcal{F},\mathcal{G}}(A)$ increases by appending the row.
For the matrix $A=X$ over the field $\mathbb{Q}(x_{i,j})$ (the degree $n^2$
transcendental extension of the field of rationales),
denote $\mathcal{F}^{<_t,X}$ by $\mathcal{F}^{<_t}$, called the exterior shifting of $\mathcal{F}$ w.r.t. $<_t$.

Kalai \cite{Kalai_ext_shift}
proved that $\mathcal{F}^{<_t}$ is shifted,
and showed that for $K$ a simplicial complex and $K_k$ its family of $k$-faces,
the union $K^{<_t,A}:=\cup_k (K_k)^{<_t,A}$ is a simplicial complex, provided that the order $S<_tT$ depends only on the symmetric difference $S \Delta T$.


Of special importance is exterior shifting with respect to the lexicographic order $<_L$, where
$S<_L T$ if $\min {S \Delta T} \in S$.
Another case, that was less studied, is exterior shifting w.r.t.
the reverse lexicographic order $<_{RL}$ defined by $S<_{RL}T$ if $\max {S \Delta T} \in T$.
(In \cite{Babson-Novik-Thomas:revlex} shifting w.r.t. $<_{RL}$ over the symmetric algebra, rather than exterior algebra, was studied.)
For the graph $C_5$ its exterior shifting w.r.t the lexicographic order gives $G_1$ and
its exterior shifting w.r.t the reverse lexicographic order gives $G_2$. It follows that $C_5$ is
weakly isomorphic to both $G_1$ and $G_2$. On the other hand Kalai \cite[Prop.4.2]{Kalai:SymmMatroids} proved that
shifted families are weakly isomorphic if and only if they are equal. This explains Example~\ref{ex:not_transitive}.


The following relation between exterior shifting w.r.t. $<_L$ and reduced homology \cite{Kalai_ext_shift, BK} is important.
\begin{lemma}\label{lem:Homology_Shifting}
For any simplicial complex $\Delta$,
$$\dim \tilde{H}_i(\Delta;\mathbb{Q})= |\{F\in \Delta^{<_L}: \ |F|=i+1,\ \{1\}\cup F \notin \Delta \}|.$$
\end{lemma}
\begin{proof}[Proof of Proposition~\ref{prop:Tree_dom}.]
From Lemma~\ref{lem:Homology_Shifting} if follows that a $k$-uniform hypergraph $G$ on $[n]$ has
$\tilde{H}_{k-2}(K(G),\mathbb{Q}) = 0$
 iff $G^{<_L}$ contains the $k$-star on $[n]$ with apex $1$, denoted by $S$,
 and $\tilde{H}_{k-1}(K(G),\mathbb{Q}) = 0$
 iff $G^{<_L} \subseteq S$.
 As domination is preserved under a permutation of $[n]$, we can assume the $k$-star in the proposition to be $S$. Both parts (i) and (ii) then follow, as an \emph{initial} segment $I \subseteq {[n] \choose k}$ w.r.t. $<_L$ dominates (resp. is dominated by) $G$ iff $G^{<_L} \subseteq I$ (resp. $I\subseteq G^{<_L}$).
\end{proof}

\subsection {Combinatorial shifting}

\begin {definition}
The \emph{combinatorial shifting} $\mathcal{F}^c$ of $\mathcal{F}$
refers to any family $\mathcal{G}$ of $k$-sets
that can be obtained from $\mathcal{F}$ by the following procedure:

(i) Pick some $1\le i<j\le n$. For all $F\in \mathcal{F}$ with $j\in F$, $i\notin F$ and $S=(F\setminus j) \cup i$ not in $\mathcal{F}$, replace $F$ by $S$, and leave the other elements of $\mathcal{F}$ intact.
Denote this operation by $\sh_{ij}$.

(ii) Repeat (i) until a shifted family is achieved.

\end {definition}

Combinatorial shifting was introduced in seminal papers by Erdos, Ko and Rado \cite{EKR} and Kleitman
\cite {Kleitman:comb_shift}, see also \cite{Frankl:comb_shift}.
As before, for a simplicial complex $K$, repeating
this procedure for all faces of $K$ (rather than of $\mathcal{F}$),
the union $K^c:= \cup_k(K_k)^c$ is a simplicial complex, with same face numbers as $K$.

\section{Tur\'{a}n with shifting}\label{sec:algebraic}
Let $<_t$ be a term order that satisfies $ab<_t a'b'$ whenever $a+b<a'+b'$.
Note that $B(n)$ is an initial segment w.r.t. $<_t$.
Recall we denote by $G^{<_t}$ the exterior shifting of $G$ w.r.t. $<_t$, and by $G^c$ a combinatorial shifting of $G$.

By the definition of exterior shifting (via a greedy choice of
a new basis w.r.t. $<_t$), and as the algebraic shifting w.r.t. any term order does not
change under a permutation of $[n]$, see \cite[Permutation Lemma]{BK}\footnote{The proof of Bj\"{o}rner and Kalai refers
to the lex order, however it extends to any term order with no difficulty.},
Theorem~\ref{thm:domination_Turan} is thus equivalent to the following statement.

\begin{theorem}[Tur\'an with shifting]\label{thm:ShiftedTuran}
Let $G$ be a graph on $[n]$ where every $3$-subset of $[n]$ contains an edge of $G$. Then $G^{<_t}$ contains $B(n)$.
\end{theorem}

The following relation between combinatorial and algebraic shifting essentially appears in the combination of Murai \cite{Murai-Join} and Murai-Hibi \cite{Murai-Hibi}.

For two simplicial complexes on $[n]$, $Z$ and $Y$, and a term order $\prec$, let $Z\prec Y$
denote that the minimal element in the symmetric difference $Z_k\triangle Y_k$ w.r.t. $\prec$ belongs to $Z$, for any $k$.
\begin{lemma}\label{lem:CobmShiftIsBigger}
For any simplicial complex $Y$, any term order $\prec$, and any combinatorial shifting,
$Y^{\prec} \prec Y^c$.
\end{lemma}

\begin{proof}
  There is a composition $Y^c=\sh_{i_1 j_1}\circ\sh_{i_2 j_2}\circ\cdots \circ\sh_{i_l j_l} Y$, with $i_s<j_s$ for all $1\leq s\leq l$.

  We show that for any simplicial complex $Z$ on $[n]$, any $1\leq i<j\leq n$, and any term order $\prec$, the following holds:

  (*) $Z^{\prec} \prec (\sh_{ij}Z)^{\prec}$.

  Then, iterating (*) we obtain the desired:\\
  $Y^{\prec} \prec (\sh_{i_l j_l}Y)^{\prec} \prec (\sh_{i_{l-1} j_{l-1}}(\sh_{i_l j_l}Y))^{\prec} \prec \cdots \prec (Y^c)^{\prec}=Y^c$,
  where the last equality follows from the fact that $Z^{\prec}=Z$ for $Z$ shifted.

  Now, (*) follows from a result of Murai, rephrased here in terms of algebraic shifting, rather than of GINs:
  for $S\subseteq [n]$, $K\subseteq 2^{[n]}$ and term order $<$, denote $m^<_S(K):=|\{T\in K:\ |T|=|S|, \ T\leq S\}|$.

  \begin{lemma}(\cite[Prop.2.4]{Murai-Join})
  For any simplicial complex $K$, any term order $<'$, any invertible $n\times n$ matrix $\phi$ and any $S\subseteq [n]$,
  $$m^{<'}_S(K^{<'}) \geq m^{<'}_S(K^{<',\phi})^{<'}).$$
  \end{lemma}
  Thus, all that is left to verify is that there exists $\phi$ s.t. $\sh_{ij}K=K^{\prec, \phi}$.
Let $\phi$ be the linear map defined by $\phi(e_k)=e_k$ for all $k\neq j$ and $\phi(e_j)=e_i+e_j$. One can check directly that this $\phi$ works, as mentioned in \cite[Lemma 1.10]{Aramova-Herzog-Hibi--Betti2000}, or apply \cite[Lemma 2.5(b)]{Murai-Hibi} to this $\phi$.
\end{proof}

\begin{lemma}\label{lem:CobmShiftContainsB}
For any graph $G$ on $n$ vertices where every $3$-subset of the vertices contains an edge of $G$, there exists a labeling of the vertex set by $[n]=\{1<2<\ldots<n\}$ and a combinatorial shifting such that $G^c$ contains $B(n)$.
\end{lemma}
\begin{proof}
  We argue by induction on $n=|V(G)|$; the cases $n=2,3$ clearly hold. Let $n\geq 4$. First we define a labeling of the vertices: if $G$ is not complete let $v_n$ and $v_{n-1}$ be nonadjacent, else label $v_1,\ldots,v_n$ arbitrarily (and the assertion is trivial). Repeat this rule for labeling the induced graph on $V(G)\setminus \{v_n,v_{n-1}\}$.

By induction, there is a combinatorial shifting
$(G[V\setminus\{v_n,v_{n-1}\}])^c$ that contains $B(n-2)$, and let $C$ be the result of applying this combinatorial shifting to $G$.

Note that any vertex $v_i$, $i<n-1$, is connected to at least one of $v_n$ and $v_{n-1}$ in $G$. Thus, also in $C$ there are at least $n-2=|B(n)|-|B(n-2)|$ edges between one of $v_n, v_{n-1}$ and the other $v_i$'s, $i<n-1$.
Hence,
$G^c=\sh_{n-2,n}(\sh_{n-3,n}(\cdots\sh_{1,n}(\sh_{n-2,n-1}(\ldots\sh_{1,n-1}(C))\cdots)$ contains $B(n)$.
\end{proof}
\begin{proof}[Proof of Theorem \ref{thm:ShiftedTuran}]
As exterior shifting is stable under permutation of the vertices, we can relabel the vertices as needed in Lemma~\ref{lem:CobmShiftContainsB} without effecting the resulted shifted graph $G^{<_t}$.
Lemmas \ref{lem:CobmShiftIsBigger} and \ref{lem:CobmShiftContainsB} imply that $G^{<_t}<_t G^c$ for some $G^c$ that contains $B(n)$. However, as $B(n)$ is an initial segment w.r.t. $<_t$, and $|E(G^c)|=|E(G)|=|E(G^{<_t})|$, also $G^{<_t}$ contains $B(n)$.
\end{proof}

\section{Tur\'{a}n with involution}\label{sec:pairs}
\begin{proof}[Proof of Theorem~\ref{thm:pairs}]
First we prove in detail the even case $n=2m$ where $\tau$ acts with no fixed points, by induction on $m$. Then we indicate the modification for all $n$ and all actions of $\tau$.

\textbf{Case: $\tau$ acts freely.}
As $\tau$ acts with no fixed points, w.l.o.g. assume $\tau(i)=n+1-i$ for all $i\in [n]$, and $n=2m$.
For the base case, one verifies the cases $n\leq 4$ by inspection.

Assume $m>2$.
\textbf{First}, assume there exists $i\in [n]$ such that the edge $\{i,\tau(i)\}\notin E(G)$. For any triple $T=\{i,\tau(i),j\}$, $T$ and $\tau(T)=\{i,\tau(i),\tau(j)\}$ support at least two edges of $G$
, thus there are at least $n-2$ edges between $V\setminus \{i,\tau(i)\}$ and $\{i,\tau(i)\}$. By induction,
$$|E(G)|\geq n-2+|E(G[V\setminus \{i,\tau(i)\}])|\geq n-2+|E(T(n-2))|=|E(T(n))|.$$

Thus, assume $\{i,\tau(i)\}\in E(G)$ for every $i\in [n]$.

\textbf{Second}, consider the case where $G$ contains an \emph{induced} matching with two edges $M=(V_M, E_M)$, $E_M=\{i\tau(i), j\tau(j)\}$.
For any vertex $k\notin V_M$ the two triples $ijk$ and $\tau(i)\tau(j)\tau(k)$ support at least 2 edges, and the
two triples $ij\tau(k), \tau(i)\tau(j)k$ support at least 2 edges, so there are at least 4 edges between $V_M$ and $\{k,\tau(k)\}$. All together, there are at least $2(n-4)$ edges in the cut $(V_M,V\setminus V_M)$.
By induction, and as the edges $i\tau(i)$ and $j\tau(j)$ exist,
$$|E(G)|\geq 2+2(n-4)+|E(G[V\setminus V_M])|\geq (n-2)+(n-4)+|E(T(n-4))|=|E(T(n))|.$$

Thus, assume further that

(*) for any $j\neq i,\tau(i)$, there is a crossing edge from $\{i,\tau(i)\}$ to $\{j,\tau(j)\}$.

\textbf{Third},
we define 3 auxiliary graphs, $W=([m],E_W)$, $Z=([m],E_Z)$ and $R=([m],E_R)$ as follows: $ij\in E_W / E_Z / E_R$ iff there exists, resp., exactly one/ exactly two/ at least $3$ edges in $G$ crossing from  $\{i,\tau(i)\}$ to $\{j,\tau(j)\}$.
Let $N_X(x)$ denote the set of neighbors of vertex $x$ in a graph $X$.
\begin{lemma}\label{lem:W}
  If $j,k\in N_W(i)$, $j\neq k$, then there exist at least 3 edges in $G$ crossing from  $\{k,\tau(k)\}$ to $\{j,\tau(j)\}$. In particular, $W$ is triangle free.
\end{lemma}
\begin{proof}[Proof of the lemma]
By interchanging the labeling of $v$ and $\tau(v)$ if needed, we can assume that either (i) $ij,\tau(i)\tau(k)\in E_W$, or (ii) $ij,ik\in E_W$. In case (i): the two triples $i\tau(j)\tau(k)$ and $\tau(i)jk$ support at least 2 edges, so $jk,\tau(j)\tau(k)\in E(G)$. Likewise, the two triples $ij\tau(k)$ and $\tau(i)\tau(j)k$ support at least 2 edges, one of them is $ij$, so at least one of $j\tau(k),\tau(j)k$ is in $G$; in total we found $3$ crossing edges from $\{k,\tau(k)\}$ to $\{j,\tau(j)\}$.

In case (ii): similarly, considering the two triples $i\tau(j)\tau(k), \tau(i)jk$ implies $jk,\tau(j)\tau(k)\in E(G)$, and considering
the two triples $ij\tau(k), \tau(i)\tau(j)k$ implies that at least one of $j\tau(k),\tau(j)k$ is in $E(G)$; in total we found $3$ crossing edges from $\{k,\tau(k)\}$ to $\{j,\tau(j)\}$ in this case as well.
\end{proof}
By definition, clearly the edge sets $E_W, E_R, E_Z$ are pairwise disjoint, so $|E(G)|\geq m+|E(W)|+3|E(R)|+2|E(Z)|$.
By Lemma~\ref{lem:W}, the graph $R$ contains the union of cliques $\cup_{v\in [m]}K_{N_W(v)}$.
Thus, the following claim, of independent interest, finishes the proof of the even case $n=2m$ (details follow):

\begin{claim}\label{claim:T-free}
Let $X$ be a triangle free graph on $m$ vertices, and define the graph $Y=Y(X)=\cup_{v\in X}K_{N_X(v)}$. Then
$$|E(Y)|+\lfloor\frac{m}{2} \rfloor \geq |E(X)|.$$
\end{claim}
\begin{remark}
The Mantel-Tur\'{a}n theorem (phrased for the complementary graph w.r.t. Theorem~\ref{thm:Turan}) easily follows from this claim, combined with the trivial inequality $|E(X)|+|E(Y)|\leq \binom{m}{2}$. Indeed, we get $2|E(X)|-\lfloor\frac{m}{2} \rfloor \leq \binom{m}{2}$, equivalently $|E(X)|\leq \lfloor \frac{m^2}{4}\rfloor$.

While in the Mantel-Tur\'{a}n theorem equality holds only for the complete bipartite graph $K_{\lfloor \frac{m}{2}\rfloor,\lceil \frac{m}{2}\rceil}$, in Claim~\ref{claim:T-free} equality holds for more graphs, e.g. for perfect matchings.
\end{remark}
\begin{proof}[Proof of the Claim]
We give an easy proof by induction on $m$.
The assertion is clear for $E(X)=\emptyset$, and for the case $m\leq 2$; assume $m\geq 3$ and $uw\in E(X)$.
As $X$ is triangle free, $N_X(u)\cap N_X(w)=\emptyset$. By the definition of $Y$, for any $u \neq w'\in N_X(w)$, and for any $w \neq u'\in N_X(u)$, we have $uw',u'w\in E(Y)$.
Thus,
$$|E(Y)|\geq |E(Y(X-\{u,w\}))|+|N_{X-u}(w)|+|N_{X-w}(u)|$$
$$
\geq |E(X-\{u,w\})|-\lfloor\frac{m-2}{2}\rfloor + |E_X(\{u,w\},V_X-\{u,w\})| = |E(X)|-1 -(\lfloor\frac{m}{2}\rfloor -1),$$
as desired, where the second inequality is by the induction hypothesis, and the first "$-1$" stands for the edge $uw$.
$E_X(A,V-A)$ stands for edges of $X$ in the cut from $A$ to its complement.
\end{proof}
Back to the proof of the free action case of Theorem~\ref{thm:pairs}, we get
$$|E(G)|\geq m+|E(W)|+3|E(R)|+2|E(Z)|\geq m+2\binom{m}{2}-\lfloor\frac{m}{2} \rfloor > |E(T(2m))|
,$$
as desired, where we used Claim~\ref{claim:T-free} for the second inequality.

\textbf{General case}.
The proof is similar to the free action case; we indicate the differences, and keep the notations from the proof of the previous case.
W.l.o.g. $n=2m+l$, $\tau(i)=2m+1-i$ for $i\in [2m]$, and $\tau(z)=z$ for $2m+1\leq z\leq n$.

In the \textbf{first} step, note that for any $i\in [2m]$, $z\in [2m+1,n]$ and $T=\{i,\tau(i),z\}$,  $\tau(T)=T$, so $T$ must support an edge of $G$, and as in the free action case we conclude there are at least $n-2$ edges crossing from $\{i,\tau(i)\}$ to $V(G)-\{i,\tau(i)\}$. So assume all edges $i\tau(i)$ exist in $G$.
By the same reasoning, also assume that any two fixed points form an edge in $G$.

In the \textbf{second} step, the triples $T=ijz$ and $\tau(T)=\tau(i)\tau(j)z$, for $z$ a fixed point, support at least 2 edges of $G$, both contain the vertex $z$, and again we conclude there are at least $2(n-4)$ edges in $G$ crossing from $\{i,j,\tau(i),\tau(j)\}$ to the complementary set of vertices. So assume there is at least one crossing edge from $\{i,\tau(i)\}$ to $\{j,\tau(j)\}$.
By similar reasoning, we can assume there is at one edge crossing from $\{i,\tau(i)\}$ to any pair of fixed points $\{z,z'\}$.

In the \textbf{third} step, define the graphs $W, R, Z$ on the vertex set $[m]$ as before, and define bipartite graphs $W'',Z'', R''$ with edges crossing from $[m]$ to $[2m+1,n]$ as follows: for $z\in [2m+1,n]$ and $i\in [m]$
$iz\in E(W'') / E(Z'') / E(R'')$ iff, resp., in $G$ there are exactly $0/1/2$ crossing edges from $z$ to $\{i,\tau(i)\}$.
Let $W'=W\cup W''$, and similarly define the graphs $Z'$ and $R'$. For a graph $H$ denote $e(H):=|E(H)|$. Thus,
\[e(W')+e(R')+e(Z')+\binom{l}{2}=\binom{m+l}{2}.\]

In order to lower bound the number of edges in $G$,
we need the following analog of Lemma~\ref{lem:W}:

for any $z\in [2m+1,n]$ and $i,j\in [m]$, $i\neq j$, we have

(a) if $i,j\in N_{W'}(z)$ then there exist all $4$ crossing edges from $\{i,\tau(i)\}$ to $\{j,\tau(j)\}$ in $G$, and

(b) if $i,z\in N_{W'}(j)$ then there exist all $2$ crossing edges from $\{i,\tau(i)\}$ to $z$ in $G$.

The verification of (a) and (b) is similar to the proof of Lemma~\ref{lem:W}.

We estimate $|E(G)|$:
\begin{eqnarray*}
e(G)\geq (m+\binom{l}{2}) +
(3e(R)+2e(Z)+e(W)) +
(2e(R'')+e(Z'')) \geq\\
\geq m+\binom{l}{2} +
(e(W')-\lfloor\frac{m+l}{2}\rfloor) +
2e(R)+2e(Z)+e(W) + e(R'')+ e(Z'') =\\
m+\binom{m+l}{2}-\lfloor\frac{m+l}{2}\rfloor+e(R)+e(Z)+e(W) =\\
m-\lfloor\frac{m+l}{2}\rfloor+\binom{m+l}{2}+\binom{m}{2}
,
\end{eqnarray*}
where in the first inequality the first summand accounts for edges from the first step and the other summands are by definition of $W',R',Z'$; the second inequality follows from combining the second and third steps with Claim~\ref{claim:T-free}, yielding
$e(R)+e(R'')\geq e(W')-\lfloor\frac{m+l}{2}\rfloor$;
and the equalities are by the definition of $W',R',Z'$ (and of $W,R,Z$).

Now, one verifies by a direct computation that the RHS is indeed $>e(T(2m+l))$.
This completes the proof of Theorem~\ref{thm:pairs}.
\end{proof}
\begin{remark}
Unlike Theorem~\ref{thm:Turan}, where the extremal example is unique, in Theorem~\ref{thm:pairs}, for $\tau(i)=n+1-i$ say, there are multiple extremal examples. In fact, following the first and second steps in the proof gives a recursive way to construct all of them, as labeled graphs. It may be of interest to characterize / count these extremal examples  up to graph isomorphism.
\end{remark}

\subsection {A general problem}

Let $\Gamma$ be a finite group acting on
$[n]$, $G$ a graph on $[n]$, and consider the condition

(**) $|\Gamma|\leq \sum_{g\in \Gamma}e(G[g(T)])$ for all $T\subseteq [n]$ with $|T|=3$.

\begin{problem}
  For which group actions $\Gamma \curvearrowright [n]$  does condition (**) imply $e(G)\geq e(T(n))$?
\end{problem}

Clearly this is the case for the trivial action of $\Gamma$, where (**) is equivalent
to the condition of Mantel-Tur\'{a}n theorem. Theorem~\ref{thm:pairs} says this is the
case for any action of $\mathbb{Z}_2$. This is not the case for the symmetric
group acting on $[n]$ by permutations, where (**) is equivalent to the weaker
bound $e(G)\geq \frac{n(n-1)}{6}$.

We can ask similar questions for other Tur\'an graphs and hypergraphs problems.
Some of those may be interesting on their own right and may shed light on the original problem.

\section {Turan's (3,4)-problem
}
\label {s:34}

Call a $3$-uniform hypergraph on $[n]$ such that any $4$ vertices span an edge a \emph{Tur\'{a}n hypergraph on $[n]$}.
First we recall Conjectures~\ref{conj:(3,4)} and \ref{conj:(3,4)domination}:
\begin{conjecture}[Tur\'{a}n's $(3,4)$-conjecture]
Let $H$ be a \emph{Tur\'{a}n hypergraph on $[n]$}.
Then $|H|\ge h(n):=
\lfloor\frac{n}{3}\rfloor\binom{n-\lfloor\frac{n}{3}\rfloor-1}{2}
$.
\end{conjecture}
\begin{conjecture}
Any Tur\'{a}n hypergraph on $[n]$ dominates $C(n):=\{abc:\ a+c\leq n,\ 2a+b\leq n,\ 1\leq a<b<c\leq n\}$.
\end{conjecture}
This last conjecture can be rephrased in terms of algebraic shifting: let $<_c$ be any term order on $[n] \choose 3$ in which $C(n)$ forms an initial segment. (Such $<_c$ exists, e.g. $\{a<b<c\}<_c\{a'<b'<c'\}$ if and only if either (i) $a+c<a'+c'$, or (ii) $a+c=a'+c'$ and $2a+b<2a'+b'$, or (iii) $a+c=a'+c'$, $2a+b=2a'+b'$ and $a<a'$.)
\begin{conjecture}\label{conj:shifting_C(n)}
If $H$ is a Tur\'{a}n hypergraph on $[n]$ then $C(n)\subseteq H^{<_c}$.
\end{conjecture}

We now relate Tur\'{a}n's Conjecture~\ref{conj:(3,4)} to exterior shifting w.r.t. the lexicographic order $<_L$.
\begin{conjecture}\label{conj:shifting_lex}
If $H$ is a Tur\'{a}n hypergraph on $[n]$ then for any $r$, $|\{F\in H^{<_L}: \ F\cap [r]\neq \emptyset\}| \ge r {n-1-r \choose 2}$.
\end{conjecture}
This conjecture is equivalent
to Kalai's \cite[Eq.(3)]{Kalai:ResearchProblem}, where algebraic shifting is not mentioned.
It is implied by Conjecture~\ref{conj:(3,4)domination}, by combining the definition of exterior shifting w.r.t. $<_L$ and a direct computation that shows, for any $r\in [n]$,
\[|\{F\in C(n):\ F\cap [r]\neq \emptyset\}|=r{n-r-1 \choose 2}.
\]

Here is another equivalent formulation of conjecture~\ref{conj:shifting_lex}, in terms of the compound matrix.

Let $H(k,r,n)=\{F\in {[n] \choose k}:\ F\cap [r]\neq \emptyset\}$, and $H(r,n)=H(3,r,n)$ for short. For a $k$-uniform hypergraph $H$ on $[n]$ let $\rm{rank}_r(H)$ be the rank of the submatrix $C_{H,H(k,r,n)}$ of the compound matrix.
\begin{conjecture}\label{conj:compound}
If $H$ is a Tur\'{a}n hypergraph on $[n]$ then for any $r$, $\rm{rank}_r(H)\ge r {n-1-r \choose 2}$.
\end{conjecture}

The case $r=\lfloor\frac{n}{3}\rfloor$, where the right hand side of the inequality is maximized, implies the original Conjecture~\ref{conj:(3,4)} by Tur\'{a}n.

The case $r=1$ of Conjecture~\ref{conj:shifting_lex} holds, as mentioned in \cite{Kalai:ResearchProblem}. Here we provide a proof.

\begin{proof}[Proof for $r=1$]
Let $\Delta=K(H)$, namely the $2$-dimensional simplicial complex consisting of $H$ and the complete graph on $[n]$, and let $\partial_i: C_i(\Delta)\rightarrow C_{i-1}(\Delta)$ be the usual boundary map from $i$-chains to $(i-1)$-chains on $\Delta$ with $\mathbb{Q}$-coefficients.
By Lemma~\ref{lem:Homology_Shifting},
$\dim H_1(\Delta):=\dim (\frac{\ker \partial_1}{\partial_2(C_2(\Delta))}) = |\{ab\in \Delta^{<_L}: \ 1ab \notin \Delta^{<_L} \}|$.
As $\dim \ker \partial_1={n-1 \choose 2}$, we need to show
\[\dim H_1(\Delta)\leq {n-1 \choose 2}-{n-2 \choose 2}=n-2
.\]
As the $1$-cycles $\{c(ab):=1a+ab-1b:\ 1<a<b\le n\}$ form a basis of the space of $1$-cycles in $\Delta$, we can choose a subset $B$ of them that forms a basis of $H_1(\Delta; \mathbb{Q})$, and let $G$ be the graph on $\{2,3,\ldots,n\}$ spanned by the edges $ab$ such that $c(ab)\in B$.
In order to show that $|B|\le n-2$ we show that
\begin{claim}
  $G$ is a forest. In particular $G$ has at most $n-2$ edges.
\end{claim}
Indeed, suppose by contradiction that for some basis $B$, $G$ contains a cycle, and let $C$ be such cycle with minimal number of vertices among all choices of $B$.
Note that for any $1<i<j<k\le n$,

(*) if  $ijk\in\Delta$ then $c(ij)+c(jk)-c(ik)=0$ in $H_1(\Delta; \mathbb{Q})$, and

(**) if $1ij \in \Delta$ then $c(ij)=0$ in $H_1(\Delta; \mathbb{Q})$.

Thus $C$ cannot have $3$ vertices $i,j,k$ as the subset $1ijk$ supports a triangle in $\Delta$.

Suppose $C$ has more than $3$ vertices, and that $i,j,k,l$ are consecutive vertices on $C$. Consider the subset $1ijk$. Then $1ij,1jk\notin \Delta$ by (**), and $ijk\notin \Delta$ by (*), as otherwise we can shorten $C$ by replacing the two edges $ij,jk$ by the chord $ik$, contradicting the minimality of $C$. Thus $1ik\in\Delta$. Similarly, considering the subset $1jkl$ gives $1jl\in\Delta$, so $c(ik)=0=c(jl)$ in $H_1(\Delta; \mathbb{Q})$.

Consider the subset $ijkl$, it spans a triangle in $\Delta$, and we show that any such triangle yields a contradiction.
If $ijk\in\Delta$ then, as $c(ik)=0$, $c(ij)$ and $c(jk)$ are homologues, a contradiction; similarly $jkl\in\Delta$ yields to contradiction.
If $ikl\in\Delta$ then, as $c(ik)=0$, $c(il)$ and $c(kl)$ are homologues, and thus replace $c(kl)$ by $c(il)$ in $B$ to get a new basis $B'$ with graph $G'$. If $C$ has only $4$ vertices then $B'$ is dependent, a contradiction, and if $C$ has more than $4$ vertices then $G'$ contains a shorter cycle obtained from $C$ by deleting the edges $ij,jk,kl$ and adding the chord $il$, contradicting the minimality of $C$.
Similarly, $ijl\in\Delta$ yields to contradiction.
\end{proof}

\bibliography{biblio1}
\bibliographystyle{plain}

\end{document}